%% file: main.tex
\documentclass[11pt,a4paper]{article}
\usepackage[T1]{fontenc}
\usepackage{lmodern,amsmath,amssymb,amsthm,microtype}
\usepackage[margin=28mm]{geometry}
\usepackage{tikz,float}
\usepackage[font=small]{caption}
\usepackage[hidelinks]{hyperref}
\newtheorem{theorem}{Theorem}
\newtheorem{lemma}[theorem]{Lemma}
\newtheorem{proposition}[theorem]{Proposition}
\newtheorem{corollary}[theorem]{Corollary}
\newtheoremstyle{compactdefinition}
  {5pt}{5pt}{\normalfont}{}{\bfseries}{.}{.5em}{}
\theoremstyle{compactdefinition}
\newtheorem{definition}{Definition}
\theoremstyle{remark}
\newtheorem*{remark}{Remark}
\theoremstyle{plain}
\newcommand{\pan}{\chi_P}
\newcommand{\bipan}{\chi_B}
\newcommand{\targetpan}{p}
\newcommand{\targetbipan}{b}
\newcommand{\targetalpha}{a}
\newcommand{\ceil}[1]{\left\lceil #1\right\rceil}
\newcommand{\floor}[1]{\left\lfloor #1\right\rfloor}
\input{figures}
\title{Counterexamples to a conjecture and improved bounds for bipanchromatic colorings}
\author{%
  \begin{tabular}[t]{@{}c@{}}
    Stefan Novakovi\'c\\
    {\small Faculty of Electrical Engineering}\\
    {\small University of Banja Luka}\\
    {\small\href{mailto:stefan.novakovic@etf.unibl.org}{stefan.novakovic@etf.unibl.org}}
  \end{tabular}
  \qquad
  \begin{tabular}[t]{@{}c@{}}
    Du\v{s}ko Joji\'c\\
    {\small Faculty of Natural Sciences and Mathematics}\\
    {\small University of Banja Luka}\\
    {\small\href{mailto:dusko.jojic@pmf.unibl.org}{dusko.jojic@pmf.unibl.org}}
  \end{tabular}}
\date{}
\hypersetup{pdftitle={Counterexamples to a conjecture and improved bounds for bipanchromatic colorings},
  pdfauthor={Stefan Novakovi\'c, Du\v{s}ko Joji\'c}}

\makeatletter
\renewcommand{\@maketitle}{%
  \begin{center}
    {\LARGE\@title\par}
    \ifx\@author\@empty\else
      \medskip{\large\begin{tabular}[t]{c}\@author\end{tabular}\par}
    \fi
    \ifx\@date\@empty\else\medskip{\large\@date\par}\fi
  \end{center}}
\makeatother
\begin{document}
\maketitle
\input{abstract}
\input{introduction}

\input{preliminaries}
\input{counterexamples}
\input{bounds}

\input{nondetermination}
\input{decidability}
\input{improved-bounds}
\section*{AI Assistance}
ChatGPT was used to assist with proof development. All AI-assisted arguments were reviewed and revised by the authors, who take full responsibility for the final manuscript.

\input{references}
\end{document}

%% file: figures.tex
\newcommand{\ParameterConstructionDrawing}{%
\begin{tikzpicture}[font=\footnotesize,
  pv/.style={circle,draw=black!75,line width=.45pt,
    minimum size=4.6mm,inner sep=0pt,font=\scriptsize},
  pe/.style={draw=black!55,line width=.5pt,rounded corners=2pt},
  pair/.style={draw=black!35,line width=.8pt},
  elabel/.style={anchor=west,inner sep=1pt,font=\scriptsize}]
  \foreach \view/\shift in {0/0,1/4.95,2/9.90}{
    \begin{scope}[xshift=\shift cm]
      \fill[black!5,rounded corners=3pt] (-1.13,-1.02) rectangle (.33,1.02);
      \node[font=\scriptsize] at (-1.30,0) {$S$};
      \draw[pe] (-1.60,-1.25) -- (.70,-1.25) -- (.70,1.28)
        -- (1.82,1.28) -- (1.82,2.25) -- (-1.60,2.25) -- cycle;
      \draw[pe,dash pattern=on 4pt off 2pt] (-1.70,1.17) -- (.80,1.17)
        -- (.80,-1.28) -- (1.90,-1.28) -- (1.90,-2.13)
        -- (-1.70,-2.13) -- cycle;
      \ifnum\view>0\relax
        \draw[pe,dash pattern=on 2pt off 1pt] (-1.50,-1.18) -- (.60,-1.18)
          -- (.60,.25) -- (1.82,.25) -- (1.82,1.05) -- (.60,1.05)
          -- (.60,2.16) -- (-1.50,2.16) -- cycle;
        \draw[pe,densely dotted] (-1.40,-1.11) -- (.50,-1.11)
          -- (.50,-.80) -- (1.82,-.80) -- (1.82,0) -- (.50,0)
          -- (.50,2.07) -- (-1.40,2.07) -- cycle;
        \node[elabel] at (1.98,.65) {$e_1$};
        \node[elabel] at (1.98,-.40) {$e_2$};
      \fi
      \node[elabel] at (1.98,1.75) {$e_0$};
      \node[elabel] at (1.98,-1.70) {$e_3$};
      \ifnum\view=2\relax
        \draw[pair] (-.45,1.65) -- (-1.95,1.65)
          -- (-1.95,-1.70) -- (-.45,-1.70);
        \draw[pair] (-.80,.72) -- (1.45,1.65);
        \draw[pair] (0,.72) -- (1.45,.65);
        \draw[pair] (0,0) -- (1.45,-.40);
        \draw[pair] (0,-.72) -- (1.45,-1.70);
        \draw[pair] (-.80,0) -- (-.80,-.72);
      \fi
      \foreach \idx/\px/\py/\col in
        {1/-.80/.72/blue!25,2/0/.72/orange!35,
         3/-.80/0/green!30,4/0/0/teal!25,
         5/-.80/-.72/yellow!35,6/0/-.72/magenta!25}{
        \ifnum\view=2\relax
          \ifnum\idx=5\relax\def\col{green!30}\fi
        \fi
        \node[pv,fill=\col] at (\px,\py) {$s_{\idx}$};
      }
      \node[pv,fill=violet!25] at (-.45,1.75) {$u$};
      \node[pv,fill=violet!25] at (-.45,-1.70) {$w$};
      \ifnum\view=2\relax
        \node[pv,fill=blue!25] at (1.45,1.75) {$x_0$};
        \node[pv,fill=orange!35] at (1.45,.65) {$x_1$};
        \node[pv,fill=teal!25] at (1.45,-.40) {$x_2$};
        \node[pv,fill=magenta!25] at (1.45,-1.70) {$x_3$};
      \else
        \node[pv,fill=black!20] at (1.45,1.75) {$x_0$};
        \node[pv,fill=black!20] at (1.45,-1.70) {$x_3$};
        \ifnum\view=1\relax
          \node[pv,fill=black!20] at (1.45,.65) {$x_1$};
          \node[pv,fill=black!20] at (1.45,-.40) {$x_2$};
        \fi
      \fi
      \ifcase\view
        \node[align=center] at (.15,2.85) {(a) Starting hypergraph};
        \node at (.15,-2.50) {$10$ vertices, $2$ hyperedges};
      \or
        \node[align=center] at (.15,2.85) {(b) Panchromatic $8$-coloring};
        \node at (.15,-2.50) {$12$ vertices, $4$ hyperedges};
      \or
        \node[align=center] at (.15,2.85) {(c) Bipanchromatic $6$-coloring};
        \node at (.15,-2.50) {$6$ pairs, each with its own color};
      \fi
    \end{scope}
  }
\end{tikzpicture}%
}

\newcommand{\GeneralSunflowerDrawing}{%
\begin{tikzpicture}[font=\footnotesize,
  v/.style={circle,draw,minimum size=4.7mm,inner sep=.2pt,font=\scriptsize}]
  \foreach \view/\shift in {0/0,1/5.15,2/10.3}{
    \begin{scope}[xshift=\shift cm,scale=.52]
      \foreach \idx/\angle/\pattern in
        {1/90/solid,2/210/densely dashed,n/330/densely dotted}{
        \begin{scope}[rotate=\angle]
          \path[draw=black!60,line width=.55pt,\pattern,
            fill=black!4,fill opacity=.3]
            (1,0) ellipse [x radius=2.7cm,y radius=2.1cm];
        \end{scope}
        \node[fill=white,inner sep=.5pt] at (\angle:4.15) {$e_{\idx}$};
      }
      \ifnum\view=1\relax
        \foreach \idx/\cx/\cy/\col in
          {1/-.999/.325/blue!25,2/0/1.05/blue!25,
           3/-.617/-.849/green!40,4/.617/-.849/green!40,n/.999/.325/blue!25}{
          \node[v,fill=\col] at (\cx,\cy) {$s_{\idx}$};
        }
        \foreach \dotangle in {328,344,360}{
          \fill (\dotangle:1.45) circle [radius=.9pt];
        }
      \else
        \foreach \idx/\cx/\cy/\col in
          {1/0/.60/blue!25,2/-.60/-.25/orange!35,n/.60/-.25/teal!25}{
          \node[v,fill=\col] at (\cx,\cy) {$s_{\idx}$};
        }
        \begin{scope}[yshift=-.35cm]
          \foreach \dotangle in {235,252.5,270,287.5,305}{
            \fill (\dotangle:.85) circle [radius=.9pt];
          }
        \end{scope}
      \fi
      \foreach \idx/\angle/\col in
        {1/90/blue!25,2/210/orange!35,n/330/teal!25}{
        \ifnum\view<2\relax
          \node[v,fill=violet!25] at (\angle:3.1) {$p_{\idx}$};
        \else
          \node[v,fill=\col] at (\angle:3.1) {$p_{\idx}$};
        \fi
      }
      \foreach \dotangle in {250,260,...,290}{
        \fill (\dotangle:3.1) circle [radius=.9pt];
      }
      \ifcase\view
        \node[align=center] at (0,5.35) {(a) $\pan$-coloring};
        \node[align=center] at (0,-4.15)
          {$\pan=n+1$, $\alpha_{n+1}=n$};
      \or
        \node[align=center] at (0,5.35) {(b) Merging-based};
        \node at (0,-4.15) {$1+\floor{\frac n2}\text{ colors}$};
      \or
        \node[align=center] at (0,5.35) {(c) Bipanchromatic $\bipan$-coloring};
        \node at (0,-4.15) {$\bipan=n$};
      \fi
    \end{scope}
  }
\end{tikzpicture}%
}

%% file: abstract.tex
\begin{abstract}
We study the bipanchromatic number $\chi_B$, introduced by Lalou et al.~\cite{LMST} as the maximum number of colors in a panchromatic hypergraph coloring where every color appears at least twice. They show that merging uniquely occurring colors in a maximum panchromatic coloring gives the lower bound $\beta=\chi_P-\lceil\alpha/2\rceil \leq \chi_B$, where $\chi_P$ is the panchromatic number and $\alpha$ is the minimum number of such colors among maximum panchromatic colorings. They conjectured that $\chi_B=\beta$ and subsequently disproved it~\cite{LMSTChordal}. 
We identify the smallest counterexample, with six vertices and three hyperedges, both minimum possible. 
We show that the ratio bound $\chi_B/\beta<2$ is asymptotically sharp and prove minimality results for the order and hyperedge count.
For $3\le\alpha<\chi_P$ we show that every integer from $\beta$ to $\chi_P-1$ occurs as $\chi_B$ while $\chi_P$ and $\alpha$ are held fixed.
Furthermore, we discover that the values $|V|,\chi_P,\alpha$ directly decide whether the conjectured equality holds except when $\alpha=4$, where a matching reduction resolves the ambiguity. 
We derive the optimal universal lower bound $L$ depending only on $|V|,\chi_P,\alpha$ in the nontrivial regime $0<2\alpha<|V|$. Finally, we prove $\chi_B/L\le5/4$ and $\chi_B/L\le6/5+1/(5L)$, with $5/4$ sharp and $6/5$ asymptotically sharp as $L\to\infty$.
\end{abstract}

%% file: introduction.tex
\section{Introduction}

A \emph{panchromatic} coloring~\cite{LMST,Lovasz1970,BPRS2013,Cherkashin2018} uses every color in every
hyperedge, with maximum $\pan$ colors.
Lalou et al.~\cite{LMST} introduced bipanchromatic colorings
to study completely independent spanning trees in split graphs.
These colorings use each color at least twice
overall, with maximum $\bipan$.
Let $\alpha$ be the minimum number of colors occurring once in a $\pan$-coloring.
Their conjecture~\cite[Conjecture~1]{LMST} predicts
$\bipan=\pan-\ceil{\alpha/2}$ by merging those colors.

The same authors subsequently disproved this conjecture~\cite{LMSTChordal}. They construct counterexamples as 
three-hyperedge sunflowers with a common set of $6d$ vertices
and three disjoint petals of size $q$, where $d\ge1$ and $q\ge2d$.
Their parameters are $\pan=6d+q$, $\alpha=6d$, and $\bipan=4d+q$.
Thus the conjectured value is exceeded by exactly $d$, on
$6d+3q\ge12d$ vertices. 

Our three-hyperedge construction uses $12d-6$ vertices for the
same excess (Proposition~\ref{prop:three-edge-excess}). Without
restricting the number of hyperedges, we prove that the minimum order
is $4d+2$ and classify all examples attaining it
(Theorem~\ref{thm:vertex-minimum}). Their family has actual-to-conjectured ratio at most $\frac65$. We determine the sharp worst-case ratio over all hypergraphs, proving that it is strictly below two and can approach two (Theorem~\ref{thm:sharp-ratio}). We develop the coloring bounds and
equality criteria suggested in their concluding
remarks~\cite[Section~5]{LMSTChordal}.

Sections~\ref{sec:preliminaries}--\ref{sec:counterexample}
introduce the colorings and counterexamples.
Section~\ref{sec:merging}
establishes sharp bounds and characterizes minimum-size counterexamples.
Section~\ref{sec:two-parameters} determines, with minimum order, all possible $\bipan$
values with fixed $\pan$ and $\alpha$, showing that
no formula depending
only on $\pan$ and $\alpha$ can determine $\bipan$
for every hypergraph. Section~\ref{sec:zero-excess} shows that the hypergraph order,
$\pan$ and $\alpha$ determine if the conjectured equality
holds, except when $\alpha=4$ which is resolved by maximum matching. In Section~\ref{sec:vca-bounds}, the optimal lower bound $L$ for $\bipan$ is determined from $|V|$, $\pan$ and $\alpha$, with
$\bipan/L\le5/4$ and $\bipan/L\le6/5+1/(5L)$,
where $5/4$ is sharp and $6/5$ is asymptotically sharp as $L\to\infty$.

%% file: preliminaries.tex
\section{Preliminaries}\label{sec:preliminaries}

A hypergraph $H=(V,\mathcal E)$ consists of a non-empty set
$V=V(H)$ (the vertices of $H$) and a hyperedge set $E(H)=\mathcal
E \subset \mathcal{P}(V)$. If every hyperedge of $H$ contains
exactly $r$ vertices, $H$ is $r$-uniform. A proper coloring of a
hypergraph assigns colors to its vertices so that no hyperedge of
size at least two is monochromatic.

\begin{definition}[Panchromatic colors]\label{def:panchromatic}
A \emph{panchromatic $k$-coloring} of $H$ assigns $k$ colors to
its vertices so that each hyperedge contains at least one vertex
of each color. The maximum such $k$ is the \emph{panchromatic
number} $\pan(H)$.
\end{definition}

\begin{definition}[Minimum unique-color count]\label{def:unique-count}
A color is \emph{unique} in a panchromatic coloring of $H$ if it
is used on exactly one vertex. For every integer $k\in[\pan(H)]$,
let $\alpha_k(H)$ denote the minimum number of vertices with
unique colors among all panchromatic $k$-colorings of $H$. By
convention, $\alpha_0(H)=0$.
\end{definition}

\begingroup
\setlength{\abovedisplayskip}{6pt plus 2pt minus 2pt}
\setlength{\belowdisplayskip}{6pt plus 2pt minus 2pt}
\setlength{\abovedisplayshortskip}{0pt plus 2pt}
\setlength{\belowdisplayshortskip}{3pt plus 2pt minus 1pt}
\setlength{\intextsep}{4pt plus 2pt minus 1pt}
\setlength{\abovecaptionskip}{4pt}
\begin{definition}[Bipanchromatic colors]\label{def:bipanchromatic}
A panchromatic $k$-coloring is \emph{bipanchromatic}~\cite{LMST}
if no color is unique. The \emph{bipanchromatic number} $\bipan(H)
$ of $H$ is the maximum $k$ such that $H$ admits a bipanchromatic
coloring. \footnote{Lalou et al.~\cite{LMST} use $\chi_p$ and
$\chi_p^2$. We write $\pan$ and $\bipan$, respectively.}
\end{definition}
Observe that $\bipan(H) = \max\{\,k\in \mathbb{N}:\alpha_k=0\,\}$.
Since every color appears on at least two vertices, it follows
that $\bipan(H)\leq\floor{|V|/2}$.

When the hypergraph $H$ is implied, we omit the argument $H$ and
write $\pan$, $\bipan$, and $\alpha_k$. Write
$\alpha=\alpha_{\pan}$. Lalou et al.~\cite[Conjecture~1]{LMST}
conjectured that
\begin{equation}\label{eq:conjecture}
 \bipan=\pan-\ceil{\frac{\alpha}{2}}.
\end{equation}
This equality was disproved in~\cite[Proposition~1]{LMSTChordal}.
For $|V|\ge2$, choose a panchromatic $\pan$-coloring with $\alpha$
unique colors. Merge the unique colors in pairs. If one unique color remains, merge it with any other color. This results in a bipanchromatic coloring with $\pan-\ceil{\alpha/2}$ colors. Thus,
$\pan-\ceil{\alpha/2}$ is a lower bound for $\bipan$. The equality
in~\cite[Conjecture~1]{LMST} would mean that this merging
procedure is optimal. The examples in Section~\ref{sec:counterexample}
have bipanchromatic colorings with more colors than this procedure produces. We define
\emph{excess} $ \epsilon=\bipan-
 \left(\pan-\ceil{\alpha/2}\right)$
\label{par:pairing-motivation} which measures the difference
between $\bipan$ and the lower bound in~(\ref{eq:conjecture}).
\endgroup

%% file: counterexamples.tex
\section{Counterexamples}\label{sec:counterexample}
The \emph{sunflower} $\mathcal{S}_n=(V,\mathcal{E})$ is an
$(n+1)$-uniform hypergraph with $2n$ vertices and $n$ hyperedges.
The set of vertices is $V=S\sqcup P$, where $S=\{s_1,\ldots,s_n\}$
and $P=\{p_1,\ldots,p_n\}$ are disjoint, and its hyperedges are
$e_i=S\cup\{p_i\}$ for all $i\in[n]$.

\begin{theorem}[Sunflower counterexamples]
\label{thm:sunflower-parameters}
For $\mathcal{S}_n$ with $n\ge3$,
\begin{equation}\label{eq:sunflower-parameters}
 \pan=n+1,\qquad \alpha=n,\qquad \bipan=n.
\end{equation}
Consequently, $\epsilon=\ceil{\frac n2}-1\ge1$, so every
$\mathcal{S}_n$ is a counterexample to~\cite[Conjecture~1]{LMST}.
In particular, for every $n\in\mathbb N$, $\mathcal{S}_{2n+1}$ has
$4n+2$ vertices, $2n+1$ hyperedges, and
$\epsilon(\mathcal{S}_{2n+1})=n$.
\end{theorem}

\begin{figure}[H]
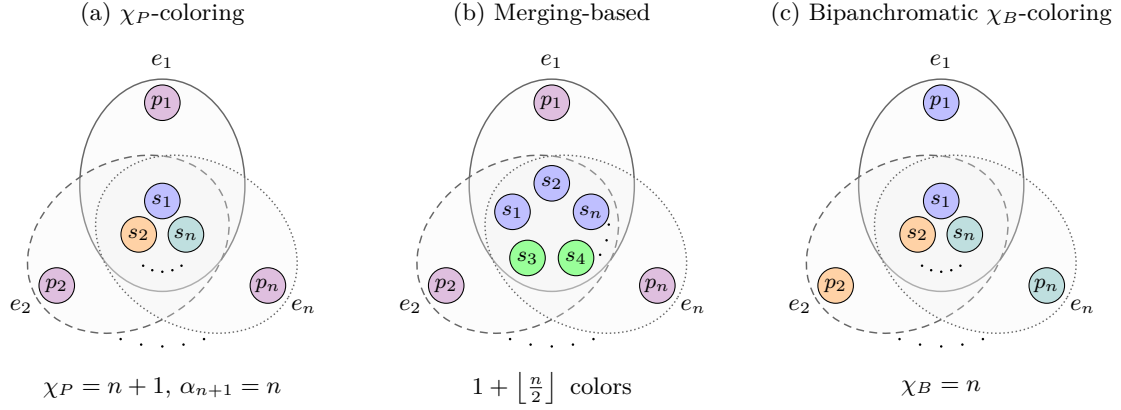

\centering
\GeneralSunflowerDrawing
\caption{$\mathcal{S}_n$ colorings. Hyperedges are $e_i=S\cup\{p_i\}$.
(a) Unique colors on $S$ and one color on $P$ give $n+1$ colors.
(b) Merging on $S$ gives $1+\floor{\frac n2}$ colors.
(c) Each $\{s_i,p_i\}$ has its own color, giving $n$ colors.}
\label{fig:color-family}
\end{figure}

\begin{proof}
We can color each $s_i$ by $i$, and all $p_i$ by $n+1$. This is a
panchromatic coloring (Figure~\ref{fig:color-family}(a)), and
$\pan=n+1$. Any panchromatic $(n+1)$-coloring of $\mathcal{S}_n$
uses each color only once per hyperedge $S\sqcup \{p_i\}$. The
colors on $S$ have to be unique, thus $\alpha=|S|=n$ and every
$p_i$ has the same color. Merging the unique colors gives
$n+1-\ceil{\frac n2}$ colors (Figure~\ref{fig:color-family}(b)).
However, we can color each $\{s_i,p_i\}$ distinctly
(Figure~\ref{fig:color-family}(c)). Each hyperedge then contains
$n$ colors, none of which is unique. Notice that
$\bipan\leq|V|/2=n$, thus $\bipan=n$ and $\epsilon=\ceil{\frac
n2}-1$.
\end{proof}
Thus, the gap between $\bipan$ and the quantity in
\cite[Conjecture~1]{LMST} is unbounded. In particular, adding any
fixed constant to the right-hand side cannot make the conjecture
valid in general.

Figure~\ref{fig:color-family}(a) shows a $\pan$-coloring of $\mathcal{S}_n$ attaining $\alpha$. Figure~\ref{fig:color-family}(b) shows a bipanchromatic coloring obtained by merging unique colors. Figure~\ref{fig:color-family}(c) shows a bipanchromatic coloring with more colors.

For excess $n$, the three-hyperedge construction
in~\cite[Proposition~1]{LMSTChordal} uses at least $12n$ vertices.
The following construction uses $12n-6$. It is obtained by
enlarging $\mathcal S_3$. Replace
each vertex by a group of $2n-1$ vertices and each hyperedge by
the union of the groups corresponding to its vertices.

\begin{proposition}[Three-hyperedge counterexamples]\label{prop:three-edge-excess}
For $n\in\mathbb N$, let $S,P_1,P_2,P_3$ be pairwise disjoint sets
with $|S|=6n-3$ and $|P_i|=2n-1$. Let $V= S\sqcup P_1 \sqcup P_2
\sqcup P_3$, $e_i={S\sqcup P_i}$, and $\Delta_n= (V,
\{e_1,e_2,e_3\})$. The excess $\epsilon(\Delta_n)$ is equal to
$n$.

\end{proposition}
\begin{proof}
Color the vertices of $S$ with distinct colors. Using $|P_i|$ additional distinct colors, color the vertices of each $P_i$ with distinct colors. This is
panchromatic, so $\pan=|e_i|=8n-4$. Any $\pan$-coloring of
$\Delta_n$ uses each color only once per hyperedge $S\sqcup P_i$,
so only the colors on $S$ are unique, thus $\alpha_{\pan}=|S|$.
Now color the vertices of $S$ with distinct colors and repeat each
such color on exactly one vertex in $P_1 \cup P_2 \cup P_3$. This
gives $\bipan=|S|$. Thus,
$\epsilon(\Delta_n)=(6n-3)-\left(8n-4-\ceil{\frac{6n-3}{2}}\right)=n$.
\end{proof}

%% file: bounds.tex
\section{Bounds and minimality}\label{sec:merging}
Consider an arbitrary hypergraph $H$ with panchromatic number
$\pan$. For any integer $k\leq \pan$, let
$\beta_k=k-\ceil{\frac{\alpha_k}{2}}$. In particular, let
$\beta=\beta_{\pan}$. \cite[Conjecture~1]{LMST} asserts that
$\bipan=\beta$, but this does not hold in general. However, for
any $k$-panchromatic coloring of $H$ attaining $\alpha_k$, the pairing argument used in
\cite{LMST} produces a bipanchromatic coloring with $\beta_k$
colors. Thus, for every $k\leq \pan$ we obtain a lower bound for
$\bipan$, and the largest of these bounds equals $\bipan$.

\begin{corollary}[Merging-based lower bound]\label{cor:pairing}
For an arbitrary hypergraph $H$ and for each integer $k\leq \pan$
we obtain that $\bipan\ge \beta_k$. Furthermore, we have that
$\bipan=\max_{k}\{\beta_k\}.$
\end{corollary}

\begin{lemma}[Equality cases and counterexample bounds]\label{lem:equality-cases}
\leavevmode

\textup{(i)} If $\alpha\in\{0,1,2,\pan\}$, then $\bipan=\beta$.
In particular, equality holds if $\pan\le3$.

\textup{(ii)} If $\bipan\ne\beta$, then
$3\le\alpha\le\pan-1$ and
$\ceil{\pan/2}\le\beta<\bipan\le\pan-1$.
\end{lemma}
\begin{proof}
\textup{(i)} Corollary~\ref{cor:pairing} gives $\bipan\ge\beta$.
If $\alpha=0$, a $\pan$-coloring is already bipanchromatic,
so $\bipan=\pan=\beta$.
If $\alpha>0$, no $\pan$-coloring is bipanchromatic, so $\bipan\le\pan-1$.
For $\alpha\in\{1,2\}$, $\bipan=\beta$.
If $\alpha=\pan$, then $|V|=\pan$, so $
\bipan\le\floor{|V|/2}
=\floor{\pan/2}
=\beta$, so $\bipan=\beta$.

\textup{(ii)} By \textup{(i)}, $3\le\alpha\le\pan-1$.
Thus $\beta\ge\pan-\ceil{\frac{\pan-1}{2}}=\ceil{\frac{\pan}{2}}$.
Corollary~\ref{cor:pairing} gives $\beta<\bipan$, and $\alpha>0$ gives
$\bipan\le\pan-1$ as above.
\end{proof}

Let $\rho=\bipan/\beta$ denote the actual-to-conjectured ratio.

\begin{theorem}[Sharp ratio bounds]\label{thm:sharp-ratio}\label{cor:factor-two}
If $|V|\ge2$, then $1\le\rho<2$, and the upper bound is
asymptotically sharp. For $\pan\ge2$, the sharper bound $\rho\le
\frac{\pan-1}{\ceil{\pan/2}}$ is attained for every fixed value of
$\pan$.
\end{theorem}
\begin{proof} By Lemma~\ref{lem:equality-cases} it holds that
$\rho\geq 1$, with equality precisely when
 (\ref{eq:conjecture})
is true. Suppose that $\pan\ge2$. If $\rho\neq1$, then
Lemma~\ref{lem:equality-cases}(ii) gives the upper bound. For
$n\ge3$, Theorem~\ref{thm:sunflower-parameters} gives
$\rho(\mathcal{S}_n)=\frac{\pan-1}{\ceil{\pan/2}}$. The sharp
bounds are strictly below $2$ and tend to $2$ as $\pan\to\infty$.
\end{proof}
\emph{The order} of a hypergraph $H$ is the number of its vertices
$|V(H)|$. We show \(\mathcal{S}_{2n+1}\) has minimum order among
hypergraphs satisfying \(\epsilon\ge n\).

\begin{theorem}[Excess-based minimum order]\label{thm:vertex-minimum}
If a hypergraph $H=(V,\mathcal E)$ satisfies $  \epsilon(H) \geq
n$, then $|V|\geq |V(\mathcal{S}_{2n+1})| = 4n + 2$. Moreover, if
$ \epsilon(H)= n$ and $|V|=4n+2$, then, after removing hyperedges
that properly contain other hyperedges,
$H\cong\mathcal{S}_{2n+1}$.
\end{theorem}
\begin{proof}
 By
Lemma~\ref{lem:equality-cases}(ii), $\beta\ge\ceil{\pan/2}$ and
$\bipan\le\pan-1$. Hence
\begin{equation}\label{eq:chain}
n\le\epsilon =\bipan-\beta \le\bipan-\ceil{\pan/2}
\le\floor{(\bipan-1)/2}.
\end{equation}
Thus, $\bipan\ge2n+1$, and consequently
$|V|\ge2\bipan\ge4n+2=|V(\mathcal{S}_{2n+1})|$.

Suppose that $ \epsilon(H)= n$ and $|V|=4n+2$. It follows from
\eqref{eq:chain} that \( \bipan=2n+1\) and \(
\beta=\ceil{\pan/2}=n+1 \). Since $\pan \neq \bipan$ ($\epsilon >
0$), we obtain $\pan=2n+2$. Moreover, \( \ceil{\frac{\alpha}2}
=\pan-\beta =n+1 \). Since $\alpha\neq\pan$ it follows that
$\alpha=2n+1$.

Consider a panchromatic $(\pan, \alpha)$-coloring of $H$, and let
$S$ be the set of its $\alpha$ uniquely colored vertices. All
remaining $|V|-\alpha$ vertices have a common color $C$. For every
vertex $v$ of color $C$, there is a hyperedge $S\sqcup\{v\}$.
Otherwise, if every hyperedge containing $v$ also contained another vertex of
$v$'s color, recoloring $v$ with an existing unique color preserves
panchromaticity and contradicts the supposed $\alpha$. Removing hyperedges that properly contain other
hyperedges leaves $\mathcal S_{2\epsilon+1}$.
\end{proof}

Every counterexample has at least three hyperedges. Choose a
$\pan$-coloring attaining $\alpha$. If every non-unique color
occurred twice, then $|V|=2\pan-\alpha$, so
$\bipan\le\floor{|V|/2}=\beta$, a contradiction.
Thus some color $C$ occurs on at least three vertices.
Since $\alpha>0$, the same recoloring argument shows that
each such vertex is the only $C$-colored vertex in some
hyperedge. These hyperedges are distinct.

Proposition~\ref{prop:three-edge-excess} attains this
three-hyperedge minimum with $\epsilon(\Delta_n)=n$ for every
$n\ge1$. However, Theorem~\ref{thm:vertex-minimum} implies
that every hypergraph of minimum order satisfying
$\epsilon(H)\ge n$ has at least $2n+1$ hyperedges.
Thus, for $n\ge2$, the minimum numbers of vertices and
hyperedges cannot be attained simultaneously.

%% file: nondetermination.tex
\section{Nondetermination}\label{sec:two-parameters}

By Corollary~\ref{cor:pairing} and Lemma~\ref{lem:equality-cases}, if
 $\pan>\alpha\geq3$, then $
\pan-\ceil{\alpha/2}\leq\bipan<\pan.$ We will show that every
integer value of $\bipan$ in this range can be realized for the
given pair $(\pan,\alpha)$. Thus, $\bipan$ cannot be expressed as
a function of $\pan$ and $\alpha$ alone.

\begin{proposition}[Realization of prescribed
coloring parameters]\label{prop:fixed-parameters} Let
$\targetpan$, $\targetalpha$, $\targetbipan$ be integers
satisfying \begin{equation}\label{e:prescribedineq}
3\le\targetalpha<\targetpan \text{ and }
 \targetpan-\ceil{\targetalpha/2}
 \le\targetbipan<\targetpan.
\end{equation} Then there exists a hypergraph satisfying
$\pan=\targetpan, \alpha=\targetalpha, \bipan=\targetbipan$.
\end{proposition}
\begin{proof}
 Start with two hyperedges $e_0$ and $e_{n+1}$,
each of size $p$, whose intersection is a set $S$ of size $a$, as
shown in Figure~\ref{fig:parameter-construction}(a). Choose
vertices $x_0\in e_0\setminus S$ and $x_{n+1}\in e_{n+1}\setminus
S$, and write \[ e_0=S\sqcup U\sqcup\{x_0\},\qquad e_{n+1}=S\sqcup
W\sqcup\{x_{n+1}\}, \] where $|U|=|W|=p-a-1$.

\noindent Since any hypergraph admitting a bipanchromatic
$b$-coloring has at least $2b$ vertices, we set
\begin{equation}\label{eq:n-new-vertices} n=\max\{0,\,2b-(2p-a)\}
\end{equation} and add $n$ new vertices $x_1,\ldots,x_n$.
 Let $
X=\{x_0,x_1,\ldots,x_n,x_{n+1}\}$.
 We then add $n$
new hyperedges $$e_i=\big(e_0\setminus \{x_0\}\big)\cup\{ x_i\},
\text{ for all } i\in[n],$$ see in
Figure~\ref{fig:parameter-construction}(b). This gives a
$p$-uniform hypergraph $H=(V,\mathcal E)$, where
\begin{equation}\label{eq:range-construction}
  V=S\sqcup U\sqcup W\sqcup X,\qquad
  \mathcal E=\{e_0,\ldots,e_{n+1}\}.
\end{equation}
By~\eqref{e:prescribedineq} and the definition of $n$,
$|V|\in\{2\targetbipan,2\targetbipan+1\}$. We define a
$p$-panchromatic coloring of $H$ as follows:
\begin{enumerate} \item[(1)] The vertices of $S$ receive
pairwise distinct colors.
\item[(2)] The vertices of $U$ and $W$ are paired,
 with each pair receiving its own color.
 This introduces $p-a-1$ additional colors.
\item[(3)] Introduce one new color. Assign that same color
to every vertex of $X$.
\end{enumerate}
 See Figure~\ref{fig:parameter-construction}(b) for details.
  Every hyperedge
contains all $\targetpan$ colors, so $\pan=\targetpan$. Moreover,
exactly the colors assigned to the vertices of $S$ are unique in
any $(\pan,\alpha)$-coloring, and hence $\alpha=\targetalpha$.

To obtain a bipanchromatic $b$-coloring of $H$
(Figure~\ref{fig:parameter-construction}(c)), we retain the colors
assigned to the vertices of $U\cup W$. Since $|X|\leq|S|$, we can
pair the vertices of $X$ with a subset of $S$. We then merge the colors of the remaining vertices of $S$ in pairs, as indicated in the figure. We recolor any leftover vertex with an existing pair's color. The
resulting coloring is bipanchromatic and uses $b$ colors. Thus,
$\bipan=\targetbipan$.
\end{proof}

The construction realizes the prescribed parameters using at most
$\alpha$ hyperedges, each of size $\pan$. Together with
Lemma~\ref{lem:equality-cases}(i), this shows that $\pan$ and
$\alpha$ uniquely determine $\bipan$ only when
$\alpha\in\{0,1,2,\pan\}$.

\begin{figure}
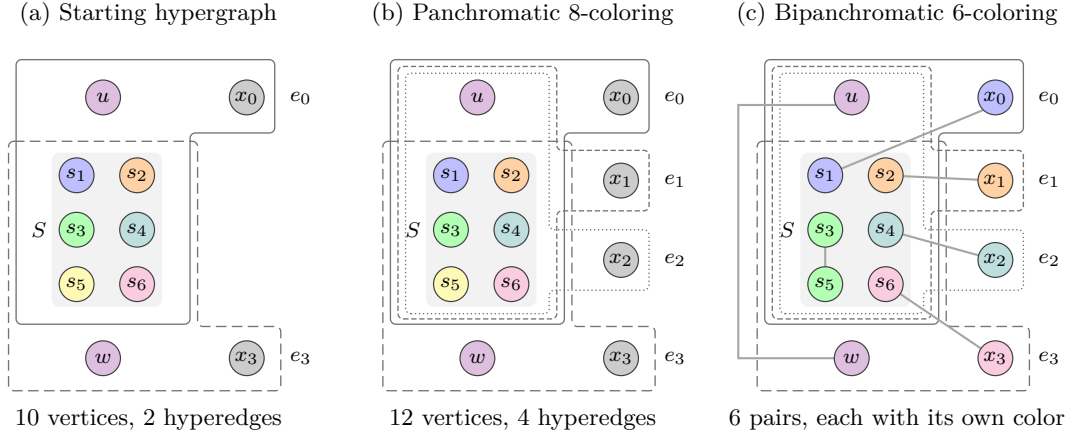

\centering \ParameterConstructionDrawing
\caption{Proposition~\ref{prop:fixed-parameters} at
$(\targetpan,\targetalpha,\targetbipan)=(8,6,6)$ and $n=2$.
$S=\{s_1,\ldots,s_6\}$, $U=\{u\}$ and $W=\{w\}$. The contours
represent the hyperedges $e_i=S\sqcup\{u,x_i\}$,
$e_{3}=S\sqcup\{w,x_{3}\}$. (a) The initial construction. (b)
Adding new vertices and hyperedges, and 8-coloring. (c)
Bipanchromatic coloring by pairing.}
\label{fig:parameter-construction}
\end{figure}

The minimum number of vertices in a hypergraph
$H$ with prescribed
feasible parameters $(\pan,\alpha,\bipan)$ is
$\max\{2\pan-\alpha,\,2\bipan\}$. Indeed, a $\pan$-coloring attaining $\alpha$ gives
$|V|\ge\alpha+2(\pan-\alpha)=2\pan-\alpha$.
A bipanchromatic $\bipan$-coloring also gives $|V|\ge2\bipan$.

For $3\le\alpha<\pan$, the construction in
Proposition~\ref{prop:fixed-parameters} has exactly
$\max\{2\pan-\alpha,2\bipan\}$ vertices.
In the remaining cases, Lemma~\ref{lem:equality-cases}(i)
gives $\bipan=\beta$.
Take two sets of size $\pan$ intersecting in $\alpha$ vertices,
use them as hyperedges, and take their union as the vertex set.
When $\alpha=\pan$, this gives just one hyperedge.
The panchromatic number is $\pan$, and every maximum
panchromatic coloring has unique colors precisely on the
$\alpha$ common vertices. The same lemma therefore gives
bipanchromatic number $\beta$.
This construction has
$2\pan-\alpha=\max\{2\pan-\alpha,2\beta\}$ vertices,
attaining the lower bound.

%% file: decidability.tex
\section{Characterization of equality}\label{sec:zero-excess}
Here we discuss conditions characterizing $\bipan=\beta$ (equivalently
$\epsilon=0$) for an arbitrary hypergraph.
\begin{proposition}[Zero excess characterization]
\label{prop:equality-order} Assume that $H$ is a hypergraph with
$\alpha\in\mathbb{N}\setminus\{1,2,4\}$. Then $\bipan=\beta$ if
and only if $ \beta=\floor{\frac{|V|}{2}}$.

\end{proposition}
\begin{proof}
Assume that $\bipan=\beta$ and consider a $(\pan,\alpha)$-coloring
of $H$.
 Let $n=3$ if $\alpha$
is odd and $n=4$ otherwise. If a color $C$ occurs at least $n$
times, recolor all vertices of color $C$ with $n$ existing unique colors, using each at least once. Merge the remaining
$\alpha-n$ uniquely colored vertices in pairs. This yields a
$\beta+1$ bipanchromatic coloring of $H$, contradicting
$\bipan=\beta$. Therefore, if $\alpha$ is odd, then every color
occurs at most twice, and if $\alpha$ is even, every color occurs
at most thrice. If two distinct colors each occur three times,
recolor their six vertices with six colors that are used only once
(since $\alpha\geq6$), and then merge the remaining uniquely
colored vertices in pairs, again obtaining a coloring with
$\beta+1$ colors. Thus, at most one color occurs thrice, so $
|V|\in\{2\pan-\alpha,\,2\pan-\alpha+1\}$, and
$\floor{|V|/2}=\beta$.

The converse follows from the definition.
\end{proof}

For a hypergraph $H$ with $\alpha=4$,
the equality $\bipan=\beta$
requires further consideration.
 Consider mutually disjoint sets $|S|=4$,
 $\{v,a,b,x_0,x_1,x_2\}$ and $|P_i|=2$
  for $i=1,2,3$. We define two hypergraphs $H_-$ and $H_+$
  whose hyperedge sets are, respectively
$$\mathcal E_-= \{S\sqcup P_i : i\in[3]\}\textrm{ and }\mathcal E_+=
\{S\sqcup\{v,x_0\},\,
  S\sqcup\{v,x_1\},\,
  S\sqcup\{v,x_2\},\,
  S\sqcup\{a,b\}\}.$$
  Notice that $|V(H_-)|=|V(H_+)|=10$,
  $\pan(H_-)=\pan(H_+)=6$, and $\alpha(H_-)=\alpha(H_+)=4$.
  However, $\beta=\bipan(H_-)=4<|V|/2$ and
  $\beta<\bipan(H_+)=5=|V|/2$.

Therefore, if $\alpha(H)=4$, we need an additional analysis. For a
given hypergraph $H=(V,\mathcal E)$, define a simple graph
$G_H=(V, E)$, where $$E=\{uv: u,v\in V, \{u,v\}\cap e \neq
\varnothing\textrm{ for all }e\in \mathcal{E}\}.$$
\begin{proposition}[Matching reduction]
\label{prop:equality-four} Let $H$ be a hypergraph with
$\alpha(H)=4$ and let $\nu$ be the size of a maximum matching of
$G_H$. Then $\bipan=\beta$ if and only if $ |V|+\nu\le3\pan-4$.
\end{proposition}

\begin{proof}
Since $\alpha=4$ and $\beta=\pan-2$, then
$\bipan\in\{\pan-2,\pan-1\}$. Consider any bipanchromatic
$k$-coloring.
Each color occurring exactly twice determines an edge of \(G_H\). These edges are vertex-disjoint, so they form a matching in \(G_H\). Therefore, we have at most $\nu$ colors used exactly two times. Consequently,
$|V|\ge2\nu + 3(k-\nu)$, so $k\le\bipan\le\floor{(|V|+\nu)/3}$.

If $|V|+\nu\le3\pan-4$, then $\bipan\le \pan-2=\beta$, and
therefore $\epsilon=0$.

For the converse, assume $\epsilon=0$ and choose a panchromatic
$\pan$-coloring attaining $\alpha=4$.

If some color $C$ occurs at least four times, we can recolor all
vertices of color $C$ with the four colors that are used only
once. This produces a bipanchromatic coloring with $\pan-1$
colors, contradicting $\epsilon=0$.

Therefore, every color is used at most three times.
 Let $\mathcal C_k$ denote
the number of colors used exactly $k$ times. Then,
$\pan=4+\mathcal C_2+\mathcal C_3$, and $|V|=4+2\mathcal
C_2+3\mathcal C_3.$

The pairs corresponding to the $\mathcal{C}_2$
 color classes must form a maximum matching in
$G_H[V\setminus S]$, where $S$ is the set of vertices receiving the
four unique colors. Otherwise, an augmenting
path would yield a larger matching, whose pairs could be assigned
distinct colors in a panchromatic coloring of $H$ (the pairs
intersect every hyperedge). The endpoints $u$ and $v$
 of such an
augmenting path belong to three-vertex color classes $\mathcal
C_3$. Reassign the colors so that each matching pair forms its own
color class. If $u$ and $v$ belong to the same three-vertex color
class, assign the color of one of the four uniquely colored
vertices to the third vertex of that class. This produces a
$\pan$-coloring with three uniquely colored vertices, a
contradiction. If $u$ and $v$ belong to different three-vertex
color classes, assign the four unique colors to their other four
vertices. This produces a bipanchromatic coloring with $\pan-1$
colors, again a contradiction, since $\bipan=\beta=\pan-2$. Thus,
$\nu\le \mathcal C_2+4$, and thus $|V|+\nu\le(4+2\mathcal
C_2+3\mathcal C_3)+(\mathcal C_2+4)=3\pan-4.$
\end{proof}
Given $H$ and promises $\pan$ and $\alpha$, $\epsilon=0$ (or
$\bipan=\beta$) is decidable in polynomial time.

%% file: improved-bounds.tex
\section{Order-aware improved bounds}\label{sec:vca-bounds}

In this section, we introduce another lower bound $L$ for the
bipanchromatic number, based on our construction, and show that
$\beta\leq L\leq\bipan$. Thus, $L$ improves the lower bound
$\beta$.

Consider a hypergraph $H=(V,\mathcal{E})$ with $|V|\ge2$ and
$\alpha>0$.
 If $\alpha\ge |V|/2$, then $ \bipan=\floor{|V|/2}$. Indeed,
 the unique vertices can be paired with the
  non-unique vertices, and any remaining vertices can
 be assigned colors so that every color used
 occurs at least twice. For the sunflower $\mathcal{S}_n$
(see Theorem~\ref{thm:sunflower-parameters}), where $n\geq3$, we
have $\alpha=|V|/2$ and $\bipan=|V|/2$, while $\bipan/\beta\to2$
as $n\to\infty$.

We now consider the case $0<\alpha<|V|/2$ and choose a
panchromatic $\pan$-coloring of $H$ with $\alpha$ unique colors.
Let $m=\pan-\alpha$ denote the number of colors used more than
once. Suppose that color $i$ is used on $c_i$ vertices, where
$$c_1\geq c_2\geq\cdots\geq c_m>1.$$
Observe that $c_1+ c_2+\cdots+ c_m=|V|-\alpha$.

 Let $d\in[m]$ be the smallest index such that
 $c_1+c_2+\cdots+c_d\geq \alpha$.
Such an integer exists because $|V|-\alpha>\alpha$.

Recolor all vertices in the first $d$ color classes with the $\alpha$
unique colors so that each of these colors is used at least twice.
This yields a bipanchromatic coloring of $H$ with $L_d=\pan-d$
colors.

Since $c_i\geq2$ for every $i$, we have $d\leq\ceil{\alpha/2}$.
Consequently, $\bipan\geq L_d=\pan-d\ge\beta$.

To obtain a bound for $d$, write Euclidean division of
$|V|-\alpha$ by $m=\pan-\alpha$ as:
\[
|V|-\alpha=qm+r,\quad\]
\[ q=\floor{\frac{|V|-\alpha}{\pan-\alpha}}\geq2,
\quad 0\le r<m.
\]
Define
\begin{equation}\label{eq:vca-d}
D=\max\left\{\ceil{\frac{\alpha}{q+1}},\,
\ceil{\frac{\alpha-r}{q}}\right\},\qquad L=\pan-D.
\end{equation}
Unlike $L_d$, which uses the chosen coloring, $L$ depends only
on $|V|$, $\pan$, and $\alpha$.

\begin{proposition}[Least class recoloring]\label{thm:vca-bounds}
Assume that the parameters of a hypergraph are defined as above.
If $0<\alpha<|V|/2$, then
\[
d\le D\le\ceil{\frac{\alpha}{2}},\qquad
\beta\le L\le L_d\le\bipan.
\]
\end{proposition}
\begin{proof}
For every $j\in[m]$, if $c_j\ge q+1$, the first $j$ classes
contain at least $j(q+1)$ vertices. Otherwise, the last $m-j$
classes contain at most $(m-j)q$ vertices, so the first $j$
contain at least $jq+r$. Thus
\[
\sum_{i=1}^j c_i\ge jq+\min\{j,r\}.
\]
The balanced sequence
\[
\widehat c_1=\cdots=\widehat c_r=q+1,\qquad
\widehat c_{r+1}=\cdots=\widehat c_m=q
\]
attains equality for every $j$. Its first $j$ terms sum to at
least $\alpha$ exactly when both $j(q+1)\ge\alpha$ and
$jq+r\ge\alpha$, so $D$ is its least such index. Therefore
$d\le D$. Since $q\ge2$ and $r\ge0$, the formula for $D$ gives
$D\le\ceil{\alpha/2}$, proving the stated bounds.
\end{proof}

\begin{remark}
If the three-hyperedge family described
in~\cite[Proposition~1]{LMSTChordal} satisfies $0<\alpha<|V|/2$,
then its parameters considered in Proposition~\ref{thm:vca-bounds} are
$q=3$, $r=0$, and $D=\alpha/3$. Thus
$L=\pan-\alpha/3=\bipan$. The boundary case $2\alpha=|V|$
is covered by the observation at the start of this section. So,
our order-aware bound gives the exact bipanchromatic number
throughout that family.
\end{remark}

Let
 $U=\min\{\pan-1,\floor{|V|/2}\}$. If $0<\alpha<|V|/2$, then
$\beta\le L\le L_d\le\bipan\le U$, and $U$ is an upper bound for
$\bipan$.

Now, we examine the relations between $U$, $\bipan$, and $L$.
\begin{theorem}[Improved bound ratio]\label{thm:vca-ratio}
The bounds $L$ and $U$ always satisfy
\begin{equation}\label{eq:vca-integer-ratio}
5U\le6L+1,\qquad 1\le\frac{\bipan}{L}\le\frac{U}{L}\le\frac54.
\end{equation}
In particular, $L\le3$ implies $U=L$. For all feasible inputs,
\begin{equation}\label{eq:vca-asymptotic-upper}
\frac{\bipan}{L}\le\frac65+\frac{1}{5L}.
\end{equation}
\end{theorem}
\begin{proof}
Use the balanced sequence $\widehat c_1,\ldots,\widehat c_m$ from
Proposition~\ref{thm:vca-bounds}. These are not
necessarily the actual color-class sizes of $H$. Their sum is
$|V|-\alpha$, and $D$ is the least index for which the first $D$
terms sum to at least $\alpha$.

Set $g=U-L\ge 0$. We prove that $L\ge 5g-1$.
Since $U\le\pan-1=L+D-1$, we obtain $g\le D-1$. By minimality of
$D$, it follows that
$$ \widehat c_D(D-1)\le\widehat c_1+\cdots+\widehat c_{D-1}\le\alpha-1.$$
Consequently,
\begin{equation}\label{eq:vca-gap-alpha}
\alpha\ge\widehat c_Dg+1.
\end{equation}
Since $\sum_{i=1}^{m}\widehat c_i=|V|-\alpha$, and
$\widehat c_i\le\widehat c_D$ for $i\ge D$,
\begin{equation}\label{eq:vca-order-upperbound}
|V| =\alpha+\sum_{i=1}^{D-1}\widehat c_i+\sum_{i=D}^{m}\widehat c_i
\le\alpha+(\alpha-1)+\widehat c_D(m-D+1).
\end{equation}
Since $U=L+g=\alpha+m-D+g$ and
$U\le|V|/2$ (by definition), inequality
\eqref{eq:vca-order-upperbound} yields
\begin{equation}\label{eq:vca-gap-tail}
(\widehat c_D-2)(m-D+1)\ge2g-1.
\end{equation}

We use $L=\alpha+m-D$ and $m-D\ge0$ to consider the following four
cases:
\[
\renewcommand{\arraystretch}{1.3}
\begin{array}{c|c|c|c}
\widehat c_D & \text{From~\eqref{eq:vca-gap-tail}} &
\text{From~\eqref{eq:vca-gap-alpha}}
& \text{Conclusion}\\
\hline 2 & g=0 & \text{--}
& L\ge5g-1\\
3 & m-D\ge2g-2 & \alpha\ge3g+1
& L\ge5g-1\\
4 & m-D\ge g-1 & \alpha\ge4g+1
& L\ge5g\\
\ge5 & \text{--} & \alpha\ge5g+1 & L\ge5g+1
\end{array}
\]

Therefore, $5g\le L+1$, which gives $5U\le6L+1$.

For $L\le3$ we obtain $U=L$. For $L\ge4$,
\[
\frac{U}{L}\le\frac65+\frac{1}{5L}\le\frac54.
\]
\end{proof}

\begin{remark}
The conclusions of Theorem~\ref{thm:vca-ratio} also hold with
$L_d$ in place of $L$. Indeed, $L\le L_d\le\bipan$, so
$5U\le6L+1\le6L_d+1$ and
$\bipan/L_d\le\bipan/L\le5/4$.
\end{remark}

We prove that for fixed parameters $(|V|,\pan,\alpha)$, without
more information, the bound $L$ is the best possible.
\begin{proposition}[Optimality]\label{prop:vca-pointwise}
For any valid parameters $(|V|,\pan,\alpha)$, with
 $0<\alpha<\frac{|V|}{2}$, there exists a hypergraph with
 these prescribed parameters that satisfies
$\bipan=L$.
\end{proposition}

\begin{proof} Compute $q$ and $r$ as before.
Set $m=\pan-\alpha$ and consider pairwise disjoint sets
$S,P_1,\ldots,P_q,Z$ with $|S|=\alpha$, $|P_i|=m$, and $|Z|=r$.
Choose $R\subseteq P_1$ with $|R|=r$.
We define the hypergraph $H=(V,\mathcal{E})$, where
\[
V=S\sqcup\bigsqcup_{i=1}^qP_i\sqcup Z, \qquad \mathcal E=
\{S\sqcup P_i:i\in[q]\} \cup\{S\sqcup(P_1\setminus R)\sqcup Z\}.
\]
Thus $|V|=\alpha+qm+r$ is attained, and every hyperedge has size
$\pan$. Now, we define a coloring of $H$ such that
    $\alpha$ vertices in $S$ receive distinct colors,
    each $P_i$ uses all $m$ remaining colors once, and,
    if $r>0$, each color from $R$ is used exactly once on $Z$. This gives a panchromatic coloring of $H$ with $\pan$
colors. In any panchromatic $\pan$-coloring of $H$, precisely
the vertices in $S$ have unique colors, so $\alpha$ is attained.

Now, we consider a bipanchromatic $k$-coloring of $H$. Assume that
precisely $s\le\alpha$ colors occur on $S$. Each of the other
$k-s$ colors occurs in every $P_i$, so at least $q$ times.
Those also absent from $P_1\setminus R$ must occur on $Z$, so
they occur at least $q+1$ times. At most $m-r$ of the $k-s$
colors occur on $P_1\setminus R$. Therefore
\[
|V|\ge2s+q(k-s),\qquad |V|\ge2s+(q+1)(k-s)-(m-r).
\]
Substituting $s\le\alpha$, $|V|=\alpha+qm+r$, and $\pan=\alpha+m$
gives
\[
q(\pan-k)\ge\alpha-r,\qquad (q+1)(\pan-k)\ge\alpha.
\]
By the definition of $D$, it follows that
\[
\chi_P-k \ge \max\left\{
\left\lceil\frac{\alpha}{q+1}\right\rceil,
\left\lceil\frac{\alpha-r}{q}\right\rceil \right\} =D.
\]
Thus every bipanchromatic coloring has at most $\pan-D=L$
colors. Since also $L\le\bipan$ by
Proposition~\ref{thm:vca-bounds}, we obtain $\bipan=L$.
\end{proof}
Therefore, at every feasible triple in this range, any universally
valid lower bound using only $(|V|,\pan,\alpha)$ is at most $L$.
Since $L$ is itself valid, it is the strongest such lower bound.
\begin{corollary}[Optimal ratio guarantees]\label{cor:vca-optimal}
For $L$, the factor $5/4$ in~\eqref{eq:vca-integer-ratio} is attained. The
constant $6/5$ in~\eqref{eq:vca-asymptotic-upper} is
asymptotically sharp as $L\to\infty$. These are optimal constants
from $|V|,\pan,\alpha$.
\end{corollary}

\begin{proof}
For any $n\in\mathbb{N}$, Proposition~\ref{prop:fixed-parameters}
describes a hypergraph whose parameters $(|V|,\pan,\alpha,\bipan)$ are $(12n-2,6n,3n+1,6n-1)$.  Here we obtain
$q=3$, $r=0$, $D=n+1$, and $L=5n-1$, so
\[
\frac{\bipan}{L} =\frac{6n-1}{5n-1} =\frac65+\frac{1}{5(5n-1)}.
\]
This is $5/4$ at $n=1$ and tends to $6/5$ as $n\to\infty$.
Proposition~\ref{prop:vca-pointwise} confirms optimality.
\end{proof}